\documentclass[a4paper,oneside,11pt]{amsart}

\reversemarginpar

\usepackage[colorlinks,hyperindex,linkcolor=blue,urlcolor=black,pdftitle={On contractions that are quasiaffine transforms 
of unilateral shifts},pdfauthor={Maria F.Gamal'}]
{hyperref}

\usepackage{amsmath}
\usepackage{amsfonts}
\usepackage{amssymb}
\usepackage{amsthm}

\usepackage[english]{babel}
\usepackage{enumerate}

\usepackage{graphicx}
\usepackage{color}

\usepackage[abbrev]{amsrefs}

\numberwithin{equation}{section}

\theoremstyle{plain}
\newtheorem{theorem}{Theorem}[section]
\newtheorem{lemma}[theorem]{Lemma}
\newtheorem{corollary}[theorem]{Corollary}

\theoremstyle{definition}
\newtheorem{remark}[theorem]{Remark}

\begin{document}

\title[On contractions]{On contractions that are quasiaffine transforms of unilateral shifts}

\author{Maria F. Gamal'}
\address{
 St. Petersburg Branch\\ V. A. Steklov Institute 
of Mathematics\\
 Russian Academy of Sciences\\ Fontanka 27, St. Petersburg\\ 
191023, Russia  
}
\email{gamal@pdmi.ras.ru}

\thanks{Partially supported by RFBR grant No. 14-01-00748-a}



\begin{abstract}It is known that if $T$ is a contraction of class $C_{10}$ and 
$I-T^\ast T$ is of trace class, 
then $T$ is a quasiaffine transform of a unilateral shift. 
Also it is known that if the multiplicity of
a unilateral shift is infinite, the converse is not true. 
In this paper the converse for a finite multiplicity
is proved: if $T$ is a contraction and $T$ is a quasiaffine transform 
of a unilateral shift of finite multiplicity,
then $I-T^\ast T$ is of trace class. As a consequence we obtain that 
if a contraction $T$ has finite multiplicity 
and its characteristic function has an outer left scalar multiple, 
then $I-T^\ast T$ is of trace class.

 Also, it is known that if a   contraction $T$ on a Hilbert space $\mathcal H$ is such that $\|b_\lambda(T)x\|\geq\delta\|x\|$ for every $\lambda\in\mathbb D$, $x\in\mathcal H$, with some $\delta>0$, and 
$$\sup_{\lambda\in\mathbb D}\|I-b_\lambda(T)^\ast b_\lambda(T)\|_{\frak S_1}<\infty$$ 
(here $b_\lambda$ is a Blaschke factor and $\frak S_1$ is the trace class of operators), then $T$ is similar to an isometry. 
In this paper the converse for a finite multiplicity
is proved: if $T$ is a contraction and $T$ is similar to an isometry of finite multiplicity, 
then $T$ satisfies the above conditions.
\end{abstract} 

\maketitle

\centerline{{\it Dedicated to the memory of Professor Yuri A. Abramovich}}

\section{Introduction}

Let $\mathcal H$ be a (complex, separable) Hilbert space, and let $T$ be 
a (linear, bounded) operator acting on 
$\mathcal H$. An operator $T$ is called a contraction if $\|T\|\leq 1$. 
It is well known that any contraction $T$ can be
uniquely decomposed  into the orthogonal sum 
$T=T_1\oplus U_{(a)}\oplus U_{(s)}$, where $T_1$ is a completely
nonunitary contraction, and $U_{(a)}$ and $U_{(s)}$ are absolutely 
continuous and singular (with respect to the 
Lebesgue measure on the unit circle) unitary operators, respectively 
(see {\cite[I.3.2]{17}}). A contraction $T$
is called {\it absolutely continuous (a.c.)}, if $U_{(s)}=\mathbb O$. 
For an a.c. contraction $T$ the 
Sz.-Nagy--Foias functional calculus is defined (see  {\cite[III.2.1]{17}}),
 that is, for any function 
$\varphi\in H^\infty$, where $H^\infty$ is the Banach algebra of 
bounded analytic functions on the open unit disk,
the operator $\varphi(T)$ acting on $\mathcal H$ is well-defined. 
An a.c. contraction $T$ is of class $C_0$
 ($T$ is a $C_0$-contraction), if there exists a function 
 $\varphi\in H^\infty$, such that $\varphi(T)=0$ and
$\varphi\not\equiv 0$. On $C_0$-contractions see \cite{2} and \cite{17}.

The {\it  multiplicity} $\mu(T)$ of an operator $T$ acting 
on a space $\mathcal H$ is the minimum dimension
of its reproducing subspaces: 
$$  \mu(T)=\min\{\dim E: E\subset \mathcal H, \ \ 
\bigvee_{n=0}^\infty T^n E=\mathcal H \} .$$

Let $T_1$ and $T_2$ be operators on spaces $\mathcal H_1$ and $\mathcal H_2$, 
respectively, and let
$X:\mathcal H_1\to\mathcal H_2$ be a linear bounded transformation such that 
$X$ {\it intertwines} $T_1$ and $T_2$,
that is, $XT_1=T_2X$. If $X$ is unitary, then $T_1$ and $T_2$ 
are called {\it unitarily equivalent}, in notation:
$T_1\cong T_2$. If $X$ is invertible, then $T_1$ and $T_2$ 
are called {\it similar}, in notation: $T_1\approx T_2$.
If $X$ a {\it quasiaffinity}, that is, $\ker X=\{0\}$ and 
$\operatorname{clos}X\mathcal H_1=\mathcal H_2$, then
$T_1$ is called a {\it quasiaffine transform} of $T_2$, 
in notation: $T_1\prec T_2$. If $T_1\prec T_2$ and 
$T_2\prec T_1$, then $T_1$ and $T_2$ are called {\it quasisimilar}, 
in notation: $T_1\sim T_2$. If $\ker X=\{0\}$, 
we write $T_1 \buildrel i \over \prec T_2$, 
while if $\operatorname{clos}X\mathcal H_1=\mathcal H_2$, we write 
$T_1 \buildrel d \over \prec T_2$.
 
It is well known and easy to see, that if 
$T_1 \buildrel d \over \prec T_2$, then $\mu(T_2)\leq \mu(T_1)$. Also
we recall that if $T_1$ and $T_2$ are unitary operators 
and $T_1\prec T_2$, then $T_1\cong T_2$ 
({\cite[II.3.4]{17}}).

It is well known that if one of $T_1$ or $T_2$ is an a.c. contraction 
and the other is a singular unitary, then the only
linear bounded transformation intertwining $T_1$ and $T_2$ is zero one.
(To see this, one can  apply the Lifting Theorem of Sz.-Nagy and Foias 
({\cite[II.2.3]{17}}) and  {\cite[Corollary 5.1]{5}}).
 Thus, if $T_1$ and $T_2$ are contractions
and $T_1=T_{1a}\oplus T_{1s}$ and $T_2=T_{2a}\oplus T_{2s}$ are
 decompositions of $T_1$ and $T_2$ such that 
$T_{1a}$ and $T_{2a}$ are a.c. contractions and $T_{1s}$ and $T_{2s}$
 are singular unitaries, then
$T_1\cong T_2$, $T_1\approx T_2$, $T_1\sim T_2$, $T_1\prec T_2$ 
if and only if
$T_{1a}\cong T_{2a}$, $T_{1a}\approx T_{2a}$, $T_{1a}\sim T_{2a}$, 
$T_{1a}\prec T_{2a}$, respectively, 
and $T_{1s}\cong T_{2s}$. 
In the sequel, we shall consider a.c. contractions.

The classes of contractions $C_{\alpha\beta}$, where 
$\alpha, \beta=\cdot, 0,1$, were introduced by Sz.-Nagy and 
Foias (see \cite{17} and references therein). Let $T$ be a contraction
 on a space $\mathcal H$. $T$ is of class $C_{1\cdot}$ 
 (a $C_{1\cdot}$-contraction), if $\lim_{n\to\infty}\|T^nx\|>0$ for
each $x\in \mathcal H$, $x\neq 0$, $T$ is of class $C_{0\cdot}$ 
(a $C_{0\cdot}$-contraction), if 
$\lim_{n\to\infty}\|T^nx\|=0$ for each $x\in\mathcal H$, and $T$ is 
of class $C_{\cdot\alpha}$, $\alpha=0,1$, if
$T^\ast$ is of class $C_{\alpha\cdot}$. Clearly, any isometry is 
of class $C_{1\cdot}$, a unitary operator is of 
class $C_{11}$, and a unilateral shift is of class $C_{10}$.

The relationships between contractions and isometries are studied by 
many authors. We mentioned here \cite{7},  \cite{10}, \cite{15}, \cite{kerchy},  \cite{16}, \cite{18}, \cite{20},
\cite{21}, \cite{22}, \cite{23}, \cite{25}, \cite{27} 
(it should be mentioned that in the last conclusion of {\cite[Theorem 2.7]{27} $V$ must be replaced by $T_1\oplus S_n$).

It is well known that a contraction $T$ is of class $C_{1\cdot}$ 
if and only if $T$ is a quasiaffine transform of 
an isometry. ``If" part is evident, and for ``only if" part we refer to
 the {\it isometric asymptote}
 $T_+^{(a)}$ of a contraction $T$, 
 see {\cite[Ch. II and IX]{17}}, {\cite[Ch. XII]{1}}, \cite{13}, and \cite{14}.
A contraction $T$ is of class $C_{11}$ if and only if $T$ is 
quasisimilar to a unitary operator {\cite[II.3.5]{17}}.
If a contraction $T$ is a quasiaffine transform of a unilateral shift, 
then $T$ is of class $C_{10}$, but the 
converse is not true, 
see, for example, \cite{11} and \cite{6}. On the other hand, if $T$ is a 
contraction of class
 $C_{10}$ and $I-T^\ast T$ is of trace class, 
then $T$ is a quasiaffine transform of a unilateral shift \cite{23},
for further results see \cite{20}, \cite{24},  and \cite{8}.
 The main result of this paper is the converse:
if a contraction $T$ is a quasiaffine transform of a unilateral 
shift $S$ and $\mu(S)<\infty$, then $I-T^\ast T$
 is of trace class 
(Theorem \ref{th3.1} below). As a previous result we should mentioned 
the following: it was proved in \cite{21} that if a contraction $T$ 
is a quasiaffine transform of a unilateral 
shift $S$ and $\mu(S)<\infty$, then 
its essentual and
 approximative point spectra 
coincide with the ones of $S$. For infinite 
multiplicity
 $\mu(S)$ it is not true, see \cite{4} or Remark \ref{rem3.4} below.

Our proof is based on the following results:

 if a contraction $T$ is similar to a unitary operator and 
$\mu(T)<\infty$, then $I-T^\ast T$ is of trace 
class \cite{16};

 if $T_1$ and $T_2$ are $C_0$-contractions and 
$T_1\prec T_2$ (then $T_1\sim T_2$), then $I-T_1^\ast T_1$
and $I-T_2^\ast T_2$ are of trace class or not simultaneously
 (see {\cite[Ch. III and VI.4.7]{2}}, or {\cite[X.5.7, X.8.8]{17}} or the original paper \cite{3});

 if a contraction $T$ is a quasiaffine transform of a unilateral
  shift $S$, then there exists a part of 
$T$ which is similar to $S$
   (a particular case of {\cite[Theorem 1]{15}, see also \cite{kerchy}, {\cite[IX.3.5]{17}}).

We shall use the following notation: $\mathbb D$ is the open unit disk, 
$\mathbb T$ is the unit circle, $H^2$ is 
the Hardy space on $\mathbb D$, 
$L^2$ is the Lebesgue space on $\mathbb T$, $H^2_-=L^2\ominus H^2$. 
For a cardinal number $n$, $0\leq n\leq\infty$, $H^2_n$, $L^2_n$, $(H^2_-)_n$ 
are orthogonal sums of $n$ copies of spaces $H^2$, $L^2$,
 $H^2_-$, 
respectively (of course, for $n=0$ the above spaces are zero ones).
 The unilateral shift $S_n$ and 
the  bilateral shift $U_n$ are the operators 
 of multiplication by the independent variable on the  spaces $H^2_n$ and
$L^2_n$, respectively. For a Borel set $\sigma\subset\mathbb T$ by 
$U(\sigma)$ we denote the operator of
 multiplication by the
independent variable on the space $L^2(\sigma)$ of functions from $L^2$ 
that are equal 
to zero a.e. on $\mathbb T\setminus\sigma$. For every 
a.c. isometry $V$ there exist cardinal numbers $k$, $\ell$, 
$0\leq k,\ell \leq\infty$, and Borel sets $\sigma_j$, 
$0\leq j-1 < \ell$, such that 
$\mathbb T\supset\sigma_1\supset\dots\supset\sigma_j\supset
\sigma_{j+1}\supset\dots$, 
 the Lebesgue measure of $\sigma_j$ is not zero, and
$$V\cong S_k\oplus \bigoplus_{j=1}^\ell U(\sigma_j)$$
(Wold decomposition). We have $\mu(V)=k+\ell$.

The paper is organized as follows. In Section 2 we consider contractions 
that are similar to an isometry. In Section 3 we consider contractions 
that are quasiaffine transforms of a unilateral shift. In Section 4 
we consider contractions that are quasisimilar to an isometry.

\section{On contractions similar to an isometry}

In this section we prove a generalization of {\cite[Theorem 4.2]{16}}. 
The first part of our proof is word-by-word the beginning of 
the proof of {\cite[Theorem 2.1]{16}}.

For $\lambda\in\mathbb D$ put $b_\lambda(z)=\frac{z-\lambda}{1-\overline\lambda z}$, $z\in\mathbb D$. 
Let $T$ be a contraction. Then $b_\lambda(T)=(T-\lambda)(I-\overline\lambda T)^{-1}$ is a contraction.  

\begin{theorem}\label{th2.1} Suppose $T$ is an a.c. contraction, 
$\mu(T)<\infty$, and $T$ is similar to an isometry.
Then  \begin{equation}\label{tag4.1}\sup_{\lambda\in\mathbb D}\|I-b_\lambda(T)^\ast b_\lambda(T)\|_{\frak S_1}
<\infty, \end{equation}where $\frak S_1$ is the trace class of operators. \end{theorem} 

\begin{proof} Let $V$ be an isometry such that $T\approx V$. 
We have $\mu(T)=\mu(V)<\infty$. Therefore,  
there exist nonnegative integers $k$, $\ell$, $0\leq k, \ell <\infty$, and 
Borel sets $\sigma_j$, $j=1, \dots, \ell$, 
such that $\mathbb T\supset \sigma_1\supset\dots\supset\sigma_\ell$ and
$$V\cong S_k\oplus\bigoplus_{j=1}^\ell U(\sigma_j).$$
We put
$$U^\prime=\bigoplus_{j=1}^\ell U(\mathbb T\setminus\sigma_j), 
\ \ \ T^\prime=T\oplus U^\prime, \ \ 
\text{ and } \ \ V^\prime=V\oplus U^\prime.$$
Then $T^\prime\approx V^\prime$ and $I-b_\lambda(T^\prime)^\ast b_\lambda(T^\prime)=
(I-b_\lambda(T)^\ast b_\lambda(T))\oplus\mathbb O$ for every $\lambda\in\mathbb D$.
 Further, $V^\prime\cong S_k\oplus U_\ell$, 
therefore, $T^\prime\approx S_k\oplus U_\ell$.
Thus, it is sufficient to prove Theorem \ref{th2.1} for a contraction $T$ 
which is similar to $S_k\oplus U_\ell$, 
where $0\leq k,\ell <\infty$.

Now we suppose that $T$ is a contraction on a space $\mathcal H$, 
$0\leq k,\ell <\infty$, 
$\mathcal K=H^2_k\oplus L^2_\ell$,  
$V=S_k\oplus U_\ell$, and $X:\mathcal K\to\mathcal H$ is a linear 
bounded invertible transformation such that $XV=TX$. Let 
$X=W(X^\ast X)^{1/2}$ be the polar decomposition 
of $X$; since $X$ is invertible, $W$ is unitary. We put $T_1=W^{-1}TW$, then
$$T\cong T_1\ \ \ \text{ and } \ \ \ (X^\ast X)^{1/2}V=T_1(X^\ast X)^{1/2}.$$
We put $A=X^\ast X$ and $B=(X^\ast X)^{1/2}$. Since $X$ is invertible, 
we have that $B$ is an invertible 
operator on $\mathcal K$. Further,  $B b_\lambda(V)= b_\lambda(T_1)B$ for every $\lambda\in\mathbb D$,  and
\begin{align*}I- b_\lambda(T_1)^\ast b_\lambda(T_1) & 
=B^{-1}BBB^{-1} - B^{-1} b_\lambda(V)^\ast BB b_\lambda(V)B^{-1} \\& =
B^{-1}(A -  b_\lambda(V)^\ast A b_\lambda(V))B^{-1}.\end{align*}
Clearly, \begin{equation}\label{tag4.2}\|I-b_\lambda(T_1)^\ast b_\lambda(T_1)\|_{\frak S_1}
\leq \|B^{-1}\|^2\|A-b_\lambda(V)^\ast Ab_\lambda(V)\|_{\frak S_1}.
 \end{equation}

Let $\lambda\in\mathbb D$ be fixed. 
First, we show that $A-b_\lambda(V)^\ast Ab_\lambda(V)$ is a positive operator. 
Sinse $T_1$ is a contraction, $b_\lambda(T_1)$ is a contraction, too. Therefore, 
$I-b_\lambda(T_1)^\ast b_\lambda(T_1)$ is a positive operator. Let $h\in\mathcal K$, then
\begin{align*}\big((A-b_\lambda(V)^\ast Ab_\lambda(V))h, h\big) & 
=\big(B(I-b_\lambda(T_1)^\ast b_\lambda(T_1))Bh, h\big) \\ &=
\big((I-b_\lambda(T_1)^\ast b_\lambda(T_1))Bh, Bh\big)\geq 0.\end{align*}
Since  $A-b_\lambda(V)^\ast Ab_\lambda(V)$ is a positive operator, 
\begin{equation}\label{tag4.3}\|A-b_\lambda(V)^\ast Ab_\lambda(V)\|_{\frak S_1}=
\sum_n\big((A-b_\lambda(V)^\ast Ab_\lambda(V))x_n, x_n\big) \end{equation}
for every orthonormal basis $\{x_n\}_n$ of $\mathcal K=H^2_k\oplus L^2_\ell$, 
the space on which $V$ and $A$ act (see, for example,  {\cite[III.8.1]{9}}).

 Since $b_\lambda(V)\cong V=S_k\oplus U_\ell$, 
there exists an orthonormal basis 
$$\chi_\lambda=\big\lbrace \ \{h_{\lambda in}\}_{n=0}^\infty,
 \ \  i=1,\dots,k, 
\ \ \ \{f_{\lambda jn}\}_{n=-\infty}^\infty, \ \ 
j=1,\dots,\ell\ \big\rbrace$$
of $\mathcal K$ such that 
$$b_\lambda(V)h_{\lambda in}=h_{\lambda, i, n+1}, \ \ n=0,1,\dots, 
\ i=1,\dots, k,$$
$$b_\lambda(V)f_{\lambda jn}=f_{\lambda, j, n+1}, \ \ n=\dots, -1,0,1,\dots,
\ j=1,\dots,\ell.$$
Put 
$$ a_{\lambda in}=(Ah_{\lambda in},h_{\lambda in}), \ \ n=0,1,\dots, 
\ i=1,\dots,k,$$
 and 
 $$ b_{\lambda jn}=(Af_{\lambda jn},f_{\lambda jn}), \ \ n=\dots,-1,0,1,\dots,
  \ j=1,\dots,\ell.$$
We have $0\leq a_{\lambda in}\leq \|A\|$, $0\leq b_{\lambda jn}\leq \|A\|$, 
$$a_{\lambda in}-a_{\lambda, i, n+1}=
\big((A-b_\lambda(V)^\ast Ab_\lambda(V))h_{\lambda in},h_{\lambda in}\big)
\geq 0$$
 and 
 $$ b_{\lambda jn}-b_{\lambda, j, n+1}=
 \big((A-b_\lambda(V)^\ast Ab_\lambda(V))f_{\lambda jn},f_{\lambda jn}\big)
 \geq 0.$$
Therefore, the sequences $\{a_{\lambda in}\}_{n=0}^\infty$ and 
$\{b_{\lambda jn}\}_{n=-\infty}^\infty$ are bounded and decreasing. Set
$$a_{\lambda i}=\lim_{n\to\infty}a_{\lambda in}, \ \ i=1,\dots,k,$$
$$ b_{\lambda j+}=\lim_{n\to\infty}b_{\lambda jn} \ \ \text{ and }
\ \ b_{\lambda j-}=\lim_{n\to -\infty}b_{\lambda jn}, \ \  j=1,\dots,\ell.$$
By \eqref{tag4.3} applyed  to the orthonormal 
basis $\chi_\lambda$, 
\begin{equation}\label{tag4.4}\begin{aligned}\|A-b_\lambda(V)^\ast & Ab_\lambda(V)\|_{\frak S_1} \\& =
\sum_{i=1}^k\sum_{n=0}^\infty(a_{\lambda in}-a_{\lambda, i, n+1})+
\sum_{j=1}^\ell\sum_{n=-\infty}^\infty(b_{\lambda jn}-b_{\lambda, j, n+1}) \\ 
&=
\sum_{i=1}^k(a_{\lambda i0}-a_{\lambda i})+
\sum_{j=1}^\ell(b_{\lambda j-}-b_{\lambda j+})\leq(k+\ell)\|A\|. \end{aligned} \end{equation}

The conclusion of  Theorem \ref{th2.1} for $T_1$ follows from \eqref{tag4.2} and \eqref{tag4.4}. 
Since $T\cong T_1$, $I-b_\lambda(T)^\ast b_\lambda(T)\cong I-b_\lambda(T_1)^\ast b_\lambda(T_1)$ for every 
$\lambda\in\mathbb D$. \end{proof} 

\begin{remark}\label{rem2.2new} By \cite{treil}, if a contraction $T$ on a Hilbert space $\mathcal H$ is such that 
$\|b_\lambda(T)x\|\geq\delta\|x\|$ for every $\lambda\in\mathbb D$, $x\in\mathcal H$, with some $\delta>0$, and 
$T$ satisfies to \eqref{tag4.1}, then $T$ is similar to an isometry. 
Theorem \ref{th2.1} shows that in the case of  finite multiplicity the converse is true.
The results from \cite{treil} are formulated in terms of the characteristic function of a contraction, see \cite{17}. 
A detailed explanation  of the relationship under consideration can be found in \cite{gamal}.\end{remark}

\begin{remark}\label{rem2.2} If a contraction $T$ is similar to an isometry, 
but $\mu(T)=\infty$, the conclusion of 
Theorem \ref{th2.1} is not true. To see this, one can take a contraction 
$T_1$, which satisfies Theorem \ref{th2.1} and 
such that   $I-T_1^\ast T_1\ne \mathbb O$, and put $T=\oplus_{n=1}^\infty T_1$. \end{remark}

\smallskip

We conclude this section by the following lemma, which will be needed in 
the sequel.
 For a proof, we refer to {\cite[VI.3.20]{2}}. Also this lemma 
can be deduced from a necessary and sufficient 
conditions 
on a positive operator be of trace class (see {\cite[III.8.1]{9}}).

\begin{lemma}\label{lem2.3} Suppose $T$ is a contraction on a space $\mathcal H$, 
and $\mathcal E\subset\mathcal H$
 is an invariant subspace of $T$, that is, 
a linear closed set such that  $T\mathcal E\subset\mathcal E$.
Let ${\mathcal E}^\perp=\mathcal H\ominus\mathcal E$, and let 
$P_{{\mathcal E}^\perp}$ be the orthogonal projection from $\mathcal H$ 
onto ${\mathcal E}^\perp$. Then $T$ has an upper triangular form
$$\begin{pmatrix} T|_{\mathcal E}& \ast \\ 0 & P_{{\mathcal E}^\perp}T|_{{\mathcal E}^\perp}
 \end{pmatrix} $$
with respect to the decomposition $\mathcal H = \mathcal E\oplus{\mathcal E}^\perp$. 

If
 $I-(T|_{\mathcal E})^\ast T|_{\mathcal E}$ and 
$I-(P_{{\mathcal E}^\perp}T|_{{\mathcal E}^\perp})^\ast 
P_{{\mathcal E}^\perp}T|_{{\mathcal E}^\perp}$
are of trace class, then $I-T^\ast T$ is of trace class. \end{lemma}

\section{On contractions that are quasiaffine transforms 
of a unilateral shift}

In this section we prove the main result of our paper. 

\begin{theorem}\label{th3.1} Suppose $T$ is a contraction, $1\leq n <\infty$, 
and $T\prec S_n$. Then $I-T^\ast T$
 is of trace class.\end{theorem}
\begin{proof} By {\cite[Theorem 1(a)]{13}} or {\cite[IX.1.3]{17}} we have
 $T_+^{(a)}\buildrel d \over \prec S_n$, therefore,  
$T_+^{(a)}\cong V\oplus S_k$, where $k\geq n$ and $V$ is an a.c. isometry. 
By {\cite[Theorem 1]{15}}, see also \cite{kerchy}, {\cite[IX.3.5]{17}},  there exists
 an invariant subspace $\mathcal E$ of $T$ 
  such that $T|_{\mathcal E} \approx S_n$. We put 
  $T_0=P_{{\mathcal E}^\perp}T|_{{\mathcal E}^\perp}$, where
$P_{{\mathcal E}^\perp}$ is the orthogonal projector on ${\mathcal E}^\perp$, 
and we shall show that $T_0$ is of class $C_0$.

For an operator $T^\prime$, by $\kappa(T^\prime)$ we denote the shift index 
of $T^\prime$:
$$\kappa(T^\prime)=\sup\{n: \ S_n\buildrel i \over\prec T^\prime\},$$ 
which was introduced in \cite{19} and studied in \cite{22} and \cite{gamshift}. By {\cite[Proposition 4]{22}},
$\kappa(T)=\kappa(T|_{\mathcal E})=n$, and, by {\cite[Corollary 1]{22}}, 
$\kappa(T)\geq\kappa(T|_{\mathcal E})+\kappa(T_0)$, therefore,  $\kappa(T_0)=0$. 
If we suppose that $T_0$ is not of class $C_0$, then, 
by {\cite[Introduction]{22}}, we must conclude that $\kappa(T_0)\geq 1$,  
a contradiction. Thus, $T_0$ is of class $C_0$.

By {\cite[Remark 2.6]{8}}, we have $\mu(T)\leq n+1$, and, since $T_0$ is 
a compression of $T$ on its coinvariant 
subspace, we have 
$\mu(T_0)\leq\mu(T)$. It is known (see {\cite[III.5.1]{2}} or {\cite[X.5.7]{17}}),
 that every  $C_0$-contraction is 
quasisimilar to a Jordan operator 
 of class $C_0$, that is, an operator of the form 
$\oplus_{j=0}^\infty S(\theta_j)$, where $\theta_j$ are inner functions 
from $H^\infty$, $\theta_{j+1}$ 
divides $\theta_j$ for all $j\geq 0$, 
and $S(\theta_j)$ is the compression of $S_1$ on its coinvariant 
subspace $H^2\ominus\theta_j H^2$; it is possible that $\theta_j\equiv 1$ 
for $j$ greater than some $j_0$.
By {\cite[III.4.12]{2}} or {\cite[X.5.6]{17}}, 
$$\mu\big(\oplus_{j=0}^\infty S(\theta_j)\big)=
\min\{ j:\  \theta_j\equiv 1 \}.$$
Let $J=\oplus_{j=0}^\infty S(\theta_j)$ be a Jordan operator such 
that $T_0\sim J$. Then 
$\mu(J)=\mu(T_0)<\infty$, therefore,  the sum 
$\oplus_{j=0}^\infty S(\theta_j)$ is actually finite,
and we have that $I-J^\ast J$ is a finite rank operator. 
By {\cite[VI.4.7]{2}}, or {\cite[X.8.8]{17}}, or \cite{3}, we conclude that
$I-T_0^\ast T_0$ is of trace class.

By Theorem \ref{th2.1}, $I-(T|_{\mathcal E})^\ast T|_{\mathcal E}$ is of trace class, 
and, by Lemma \ref{lem2.3}, we conclude that $I-T^\ast T$ is of trace class.  \end{proof}
\smallskip

\begin{corollary}\label{cor3.2} Let $T$ be a contraction, and let 
$1\leq n <\infty$. The following are equivalent:

$(1)$ $T\prec S_n$;

$(2)$ $T$ is of class $C_{10}$, $\dim\ker T^\ast =n$, and $I-T^\ast T$ 
is of trace class;

$(3)$ $T$ is of class $C_{10}$, $\dim\ker T^\ast=n$, and the 
characteristic function of $T$ has a left 
scalar multiple.\end{corollary}

(For the characteristic function of a contraction  we refer to \cite{17}.)

\begin{proof} The equivalence $(1)\Longleftrightarrow (3)$ 
is contained in \cite{21}, and the implication 
$(2)\Longrightarrow (1)$ is contained in \cite{23}. The implication 
$(1)\Longrightarrow (2)$ follows from
\cite{21} and Theorem \ref{th3.1}.  \end{proof}

\begin{corollary}\label{cor3.3} If $T$ is from Theorem \ref{th3.1}, then $I-TT^\ast$ 
is  of trace class.\end{corollary}
\begin{proof}  By {\cite[VI.3.3]{2}}, for any contraction $T$ the equality
$$\operatorname{trace}(I-T^\ast T)+\dim\ker T^\ast=
\operatorname{trace}(I-TT^\ast)+\dim\ker T$$
holds. Since $\dim\ker T^\ast<\infty$, Corollary \ref{cor3.3} follows from 
Theorem \ref{th3.1} and the above equality.  \end{proof}

\begin{remark}\label{rem3.4} There exist contractions $T$ such that $T$ 
satisfy to one of the following conditions 
({\it which can not be fulfilled simultaneously}):

(1) $T\prec S_\infty$;

(2) $\mu(T)<\infty $ and $T\sim U$, where $U$ is an a.c. unitary operator;

\noindent and $I-T^\ast T$ is not compact.

To show this, we use the following known fact. Let $T$ be a contraction 
of class $C_{1\cdot}$, and let
 $I-T^\ast T$ be compact. 
Then $T-\lambda I$ is left invertible  for any $\lambda\in\mathbb D$, 
and if
 $T$ is of class $C_{11}$, then $\sigma(T)\subset\mathbb T$ 
(by $\sigma(T)$ we denote the spectrum of 
an operator $T$). 
The proof is contained, for example, in {\cite[the end of Section 2]{20}} 
(although formally in \cite{20} $I-T^\ast T$ is of trace class, 
the proof is the same for a compact
 $I-T^\ast T$). 
An example of a contraction $T$ such that $T$ satifies
  (1) and $T$ is not left invertible
 is contained in \cite{4}. 
  An example of a contraction $T$ such that $T$
 satisfies (2) and 
$\sigma(T)=\operatorname{clos}\mathbb D$ 
 is contained in {\cite[Example 12]{12}}.\end{remark}

\section{On contractions quasisimilar to an isometry}

The following theorem gives a sufficient condition to $I-T^\ast T$ 
be of trace class, if a contraction $T$ is quasisimilar to an isometry.

\begin{theorem}\label{th4.1} Suppose $T$ is an a.c. contraction, 
$V$ is an a.c. isometry, $T$ and $V$ act on 
spaces $\mathcal H$ and $\mathcal K$, respectively, $\mu(T)<\infty$ 
and $\delta\in H^\infty$ is an outer 
function. 
Further, suppose $X:\mathcal H\to\mathcal K$ and $Y:\mathcal K\to\mathcal H$ are 
quasiaffinities such that $XT=VX$, $YV=TY$, $YX=\delta(T)$ and 
$XY=\delta(V)$. Then $I-T^\ast T$ 
is of trace class.\end{theorem}
\begin{proof} Since $T\sim V$, we have $\mu(T)=\mu(V)<\infty$. 
Therefore, there exist nonnegative integers  
$k$, $\ell$, $0\leq k,\ell<\infty$, 
and Borel sets $\sigma_j$, $j=1, \dots, \ell$, such that 
$\mathbb T\supset\sigma_1\supset\dots\supset\sigma_\ell$ and
$$V\cong S_k\oplus\bigoplus_{j=1}^\ell U(\sigma_j).$$
We put 
$$U^\prime=\bigoplus_{j=1}^\ell U(\mathbb T\setminus\sigma_j), 
\ \ T^\prime=T\oplus U^\prime, 
\ \ V^\prime=V\oplus U^\prime, 
\ \ X^\prime=X\oplus\delta(U^\prime), \ \ Y^\prime=Y\oplus I.$$
Then $T^\prime$, $V^\prime$, $X^\prime$, $Y^\prime$ satisfy 
the conditions of Theorem \ref{th4.1}, and
 $I-T^\ast T$ is of trace class 
if and only $I-T^{\prime\ast}T^\prime$ is of trace class. Further,
$V^\prime\cong S_k\oplus U_\ell$, and we can replace $V^\prime$ 
by $S_k\oplus U_\ell$. Thus, 
it is sufficient to prove Theorem \ref{th4.1} 
for $V=S_k\oplus U_\ell$, where $0\leq k, \ell<\infty$.

Now we suppose $V=S_k\oplus U_\ell$ and $\mathcal K=H^2_k\oplus L^2_\ell$, 
where $0\leq k,\ell<\infty$.
We put $\mathcal E=\operatorname{clos}Y(H^2_k\oplus H^2_\ell)$, 
where $H^2_\ell\subset L^2_\ell$.
Clearly, $T\mathcal E\subset\mathcal E$. We have
\begin{align*}\operatorname{clos}X\mathcal E & =
\operatorname{clos}X\operatorname{clos}Y(H^2_k\oplus H^2_\ell)=
\operatorname{clos}XY(H^2_k\oplus H^2_\ell) \\&
=\operatorname{clos}\delta(V)(H^2_k\oplus H^2_\ell)=
H^2_k\oplus H^2_\ell, \end{align*}
because  $\delta$  is  outer.  Thus,  we  conclude that  
$T|_{\mathcal E}\prec S_{k+\ell}$, and, by Theorem \ref{th3.1}, 
$I-(T|_{\mathcal E})^\ast T|_{\mathcal E}$ is of trace class.

We shall show that 
$$\operatorname{clos}Y^\ast{\mathcal E}^\perp=\{0\}\oplus(H^2_-)_\ell.$$
First, let $f\in{\mathcal E}^\perp$, and let $h\in H^2_k\oplus H^2_\ell$. 
Then $(Y^\ast f,h)=(f,Yh)=0$,
therefore,  $Y^\ast{\mathcal E}^\perp\subset\{0\}\oplus(H^2_-)_\ell$. 
Further, let $g\in (H^2_-)_\ell$
be such that $(Y^\ast f, 0\oplus g)=0$ for every 
$f\in{\mathcal E}^\perp$. Then $Y(0\oplus g)\in\mathcal E$,
therefore, 
$$\delta(V)(0\oplus g)=XY(0\oplus g)\in H^2_k\oplus H^2_\ell.$$
Since $V(\{0\}\oplus L^2_\ell)\subset \{0\}\oplus L^2_\ell$, 
we have $\delta(V)g\in H^2_\ell$.
But $\delta$ is outer, and we conclude that $g=0$.

Thus,
 $$T^\ast|_{{\mathcal E}^\perp}\prec V^\ast|_{\{0\}\oplus(H^2_-)_\ell}, $$
where the relation $\prec$ is realized by $Y^\ast|_{{\mathcal E}^\perp}$. 
Since $V^\ast|_{\{0\}\oplus (H^2_-)_\ell}\cong S_\ell$, 
by Corollary \ref{cor3.3} we conclude that 
$I-(T^\ast |_{{\mathcal E}^\perp})(T^\ast|_{{\mathcal E}^\perp})^\ast$
is of trace class. Further, 
$$I-(T^\ast|_{{\mathcal E}^\perp})(T^\ast|_{{\mathcal E}^\perp})^\ast=
I-(P_{{\mathcal E}^\perp}T|_{{\mathcal E}^\perp})^\ast
(P_{{\mathcal E}^\perp}T|_{{\mathcal E}^\perp}),$$
where $P_{{\mathcal E}^\perp}$ is the orthogonal projection from $\mathcal H$ 
onto ${\mathcal E}^\perp$.
Finally, we apply Lemma \ref{lem2.3} to a triangulation of $T$ with respect 
to the decomposition  
$\mathcal H=\mathcal E\oplus {\mathcal E}^\perp$, 
and we conclude that $I-T^\ast T$ is of trace class.  \end{proof}

\begin{corollary}\label{cor4.2} Suppose $T$ is an a.c. contraction,
 $\mu(T)<\infty$, and the characteristic function
 of $T$ has an outer left scalar multiple. 
 Then $I-T^\ast T$ is of trace class.\end{corollary}

(For the characteristic function of  a contraction we refer to \cite{17}.)

\begin{proof} Let $\delta$ be an outer left scalar multiple
  of the characteristic function of $T$, and let 
$V=T_+^{(a)}$ be the isometric asymptote of $T$. By {\cite[Theorem 4.14]{10}},
 $T$, $V$, and $\delta $ satisfy
 to the conditions of Theorem \ref{th4.1}. \end{proof}

\begin{remark}\label{rem4.3}  If $T$ and $\delta$ satisfy to the conditions of 
Theorem \ref{th4.1}, then $\delta$ is a left 
scalar multiple of the 
characteristic function of $T$. Proof is contained in {\cite[Lemma 4.16]{10}}.\end{remark}

\begin{remark}\label{rem4.4} There exist contractions $T$ such that $T\sim S_1$ and 
$$\sup_{\lambda\in\mathbb D}\|I-b_\lambda(T)^\ast b_\lambda(T)\|_{\frak S_1}=\infty$$
(sf. Theorem \ref{th2.1}), see \cite{treil} and {\cite[Lemma 5.9, Corollary 5.10 and Remark after it]{gamal}}.\end{remark}

\end{document}